\documentclass[11pt]{article}
\usepackage{amsmath,amssymb,amsthm,mathtools}
\usepackage{booktabs}
\usepackage[T1]{fontenc}
\usepackage{microtype}
\usepackage[hidelinks]{hyperref}
\usepackage[margin=1.15in]{geometry}

\newtheorem{theorem}{Theorem}[section]

\newtheorem{lemma}[theorem]{Lemma}
\newtheorem{corollary}[theorem]{Corollary}
\newtheorem{remark}[theorem]{Remark}
\newtheorem{example}[theorem]{Example}

\newcommand{\F}{\mathbb F}
\newcommand{\C}{\mathbb C}
\newcommand{\Q}{\mathbb Q}

\newcommand{\supp}{\operatorname{supp}}

\title{Sharp Total-Support Uncertainty for the Standard Complete Set of Mutually Unbiased Bases}
\author{Dongwei Li\\
School of Mathematics, Hefei University of Technology\\
Hefei 230601, China\\
\texttt{dongweili@hfut.edu.cn}}
\date{}

\hypersetup{
 pdftitle={Sharp Total-Support Uncertainty for the Standard Complete Set of Mutually Unbiased Bases},
 pdfauthor={Dongwei Li},
 pdfsubject={Finite Fourier uncertainty and mutually unbiased bases},
 pdfkeywords={uncertainty principles, finite Fourier analysis, finite parabola, mutually unbiased bases, extremizer classification}
}

\begin{document}
\maketitle

\begin{abstract}
We determine the sharp total-support uncertainty for the standard complete set of mutually unbiased bases in every prime dimension at least five, proving that the minimum is \(T_s(p)=p^2-p+2\). This resolves all prime dimensions left open by Fiorentino and Weigert, beginning with dimension eleven, and gives a uniform analytic determination of the optimal lower bound in dimension seven, where the earlier determination was partly numerical. The main new ingredient is a sharp zero-count theorem for arbitrary complex Fourier data supported on the finite parabola. A cyclotomic reduction and truncated local expansion produce a first nonzero finite-field jet that is affine in the transverse variable, reducing the planar zero count to a one-dimensional multiplicity problem. We then classify all extremizing projective states and derive an exact enumeration formula. In dimension seven, this gives analytic four-point extremizers and exactly \(12348\) such rays, completing the analytic description and count left open in the earlier work.
\end{abstract}



\section{Introduction}
\label{sec:introduction}

Fiorentino and Weigert introduced a total-support uncertainty problem for the standard complete set of mutually unbiased bases (MUBs) in prime dimension. They derived a general lower bound and reported sharp values through dimension seven, with the dimension-seven analysis partly numerical; they also noted that analytic formulas and the total number of four-point extremizers in dimension seven were not known. The first unresolved prime dimension is therefore eleven \cite{FW}. We determine the exact value for every prime dimension at least five and give a complete analytic classification and count of all extremizing projective states. The problem is a genuinely multi-basis analogue of finite Fourier uncertainty: it asks how many measurement amplitudes must remain nonzero when one state is expanded simultaneously in all bases of a maximal pairwise-unbiased family.

Let $p\ge5$ be prime and put $\omega=e^{2\pi i/p}$. Write $\mathcal B_\infty$ for the computational basis of $\C^p$, and for $a\in\F_p$ let
\[
 \phi_{a,b}(x)=p^{-1/2}\omega^{ax^2+bx},\qquad b,x\in\F_p,
\]
with $\mathcal B_a=\{\phi_{a,b}:b\in\F_p\}$. These $p+1$ bases form the standard complete set of MUBs \cite{WF,BandyopadhyayMUB,KlappeneckerRottelerConstruction}. For $0\ne\psi\in\C^p$, define
\begin{equation}
 \mathcal S_p(\psi):=\sum_{j\in\F_p\cup\{\infty\}}|\supp_{\mathcal B_j}(\psi)|,
 \qquad
 T_s(p):=\min_{\psi\ne0}\mathcal S_p(\psi).
 \label{eq:intro-Sp-Ts}
\end{equation}
Here $\supp_{\mathcal B_j}(\psi)$ is the set of nonzero coordinates of $\psi$ in the basis $\mathcal B_j$. The main result is the following.

\begin{theorem}[Exact complete-MUB support uncertainty]
\label{thm:intro-exact}
For every prime $p\ge5$,
\[
 T_s(p)=p^2-p+2.
\]
Moreover, every extremizing projective state is classified explicitly in Theorem~\ref{thm:complete-extremizer-classification}, and the extremizing rays are enumerated in Corollary~\ref{cor:complete-extremizer-count}.
\end{theorem}

The theorem also has a direct measurement interpretation. Across the standard complete set there are \(p(p+1)\) basis-outcome labels, and a label has probability zero exactly when the corresponding amplitude vanishes. Hence Theorem~\ref{thm:intro-exact} is equivalent to
\[
 p(p+1)-\mathcal S_p(\psi)\le 2p-2.
\]
Thus a pure state can make at most \(2p-2\) outcomes strictly impossible across all \(p+1\) standard MUB measurements, and the extremizing states are exactly those attaining this maximum number of zero-probability outcomes. In particular, the first previously unresolved case has
\[
 T_s(11)=112.
\]

This sharp value is stronger than what follows by summing pairwise uncertainty inequalities. For two distinct standard MUBs, the transition matrix is equivalent to the prime Fourier matrix, so Tao's prime cyclic uncertainty principle \cite{Tao} gives
\[
 |\supp_{\mathcal B_j}(\psi)|+|\supp_{\mathcal B_k}(\psi)|\ge p+1\qquad(j\ne k).
\]
Fiorentino and Weigert obtained the general lower bound
\begin{equation}
 \mathcal S_p(\psi)\ge\frac{(p+1)^2}{2}
 \label{eq:intro-FW-bound}
\end{equation}
and reported the sharp values through dimension seven \cite{FW}. The sharp complete-MUB problem, however, depends on the common quadratic-phase structure of all $p$ finite bases and cannot be reduced to a single pair of bases.

The key new ingredient is a zero-count theorem for sparse Fourier sums supported on the finite parabola. For $R\subset\F_p$ and nonzero coefficients $c_x$, set
\begin{equation}
 F_{R,c}(a,b)=\sum_{x\in R}c_x\omega^{ax^2+bx},\qquad (a,b)\in\F_p^2,
 \label{eq:intro-parabolic-sum}
\end{equation}
and let $Z(F_{R,c})$ denote its zero set. Up to normalization and signs, $F_{R,c}$ is the inverse Fourier transform of a function supported on
\[
 \Gamma=\{(x^2,x):x\in\F_p\}\subset\F_p^2.
\]
The even-support estimate is sharp throughout its full range: for every even \(k\) with \(2\le k\le p-1\), equality is attained; see Corollary~\ref{cor:even-parabolic-sharpness}.

\begin{theorem}[Parabolic zero theorem]
\label{thm:intro-parabolic}
Let $R\subset\F_p$ satisfy $1\le k:=|R|\le p-1$, and let $c_x\ne0$ for all $x\in R$. If $k=1$, then $F_{R,c}$ has no zeros. If $k\ge2$, then
\[
 |Z(F_{R,c})|\le p+2\left\lfloor\frac{k}{2}\right\rfloor-2.
\]
\end{theorem}

The passage from Theorem~\ref{thm:intro-parabolic} to Theorem~\ref{thm:intro-exact} is direct. After moving a minimum-support MUB of an extremizer candidate to the computational basis, suppose that the computational support has size $k$. The zero coefficients in the remaining $p$ quadratic bases are precisely the zeros of a sum of the form \eqref{eq:intro-parabolic-sum}. Theorem~\ref{thm:intro-parabolic} therefore gives
\[
 \mathcal S_p(\psi)\ge p^2-p+k-2\left\lfloor\frac{k}{2}\right\rfloor+2\ge p^2-p+2.
\]
Two-point states attain equality.

The proof of Theorem~\ref{thm:intro-parabolic} is algebraic rather than norm-based. A rank argument first reduces arbitrary complex coefficients either to a smaller support or to a one-dimensional kernel defined over the cyclotomic field $\Q(\omega)$. Localizing at $\pi=1-\omega$ and passing to a truncated expansion then produces finite-field jets $Q_j(a,b)$ satisfying the recurrence
\[
 \partial_b^2Q_j=\partial_aQ_{j-1}.
\]
A Vandermonde moment argument bounds the order of the first nonzero jet, while the recurrence forces that jet to be affine in the transverse variable $b$. The planar zero count then reduces to a one-dimensional multiplicity estimate for the common roots of the coefficient polynomials. This mechanism is specific enough to give the exact cardinality estimate required by the MUB problem, rather than only a support or norm inequality.

Equality in the MUB bound forces equality at every stage of this zero-count argument. This yields a reflection structure for the minimum-support coordinates and a corresponding antisymmetry of the dephased coefficients. The remaining zeros are encoded by a finite constraint matrix. We prove that every maximal minor of this matrix is nonzero, obtain a converse construction for every admissible datum, and deduce both the possible minimum-support sizes and an exact enumeration formula. Thus the zero theorem is not merely a device for proving the numerical lower bound: its equality mechanism determines all extremizers.

The finite-parabola viewpoint sits between several established themes. Support uncertainty on finite groups has a broad history \cite{DonohoStark,SmithGroups,Meshulam,WigdersonSurvey}. In prime cyclic groups, the nonvanishing of Fourier minors underlies the sharp additive principle \cite{EvansIsaacs,Tao,LoukakiChebotarev}; support uncertainty for pairs of bases and related incompatibility notions has also been developed from several viewpoints \cite{EladBruckstein,DeBievre,DeBievre2023,XuIncompatibility}. On the geometric side, finite-plane uncertainty and finite-field restriction theory show how the geometry of a spectral set can strengthen uncertainty phenomena \cite{BiroLev,GarciaKaraaliKatz,MockenhauptTao,Lewko,LewkoKakeya,IKL,IosevichMayeli,IosevichJamingMayeli}. Since \(|Z(F)|=p^2-|\supp F|\) on \(\F_p^2\), the distinction here is not between zero counting and support counting as quantities. What is needed is a coefficient-uniform sharp bound for Fourier transforms whose spectrum is constrained to the specific curved set \(\Gamma\), with a parity-dependent constant, together with the fiberwise equality structure. The cited uncertainty and restriction estimates provide important neighboring results, but they do not by themselves identify this exact bound or the equality configuration required simultaneously across the complete MUB family. MUB constructions and their projective-design structure are classical in quantum information and finite frame theory \cite{WF,BandyopadhyayMUB,KlappeneckerRottelerConstruction,DurtMUB,McNultyWeigertReview,KlappeneckerRottelerDesign,RoyScottDesign}. Thus the decisive feature is the geometry of the zero set of \(F\) under a sparse quadratic-phase spectral constraint, which makes the complete-MUB extremal problem tractable.

The equality problem is also an essential part of the main result. Complete equality classifications are known in other finite uncertainty settings \cite{BonamiGhobber,NicolaSTFT}; here the equality geometry is instead forced by the curved spectral constraint and must remain compatible across all $p+1$ bases.

\paragraph{Relation to the arXiv version.}
A broader preprint version of this work appeared as \cite{LiParabolaArxiv}. The present manuscript refocuses that version on the complete-MUB uncertainty problem and retains only the finite-parabola theory needed for the sharp bound, equality analysis, classification, and enumeration.

\paragraph{Organization.}
Section~\ref{sec:setup} fixes the MUB notation and the required symmetry. Section~\ref{sec:parabolic-exact} proves the finite-parabola zero theorem. Section~\ref{sec:consequences} derives the exact complete-MUB uncertainty bound. Sections~\ref{sec:extremizer-structure} and~\ref{sec:complete-extremizers} prove the rigidity, classification, and enumeration of all extremizers. A short conclusion summarizes the mechanism and the scope of the result.

\section{Preliminaries}
\label{sec:setup}

Let $p\ge5$ be prime, let $\F_p$ be the field with $p$ elements, and put
\[
 \omega=e^{2\pi i/p}.
\]
For $a,b\in\F_p$, define
\begin{equation}
 \phi_{a,b}(x)
 =p^{-1/2}\omega^{ax^2+bx},
 \qquad x\in\F_p,
 \label{eq:standard-MUB}
\end{equation}
and let
\[
 \mathcal B_a=\{\phi_{a,b}:b\in\F_p\}.
\]
Together with the computational basis
$\mathcal B_\infty=\{e_x:x\in\F_p\}$, these form the standard complete set
of $p+1$ MUBs \cite{WF}.

For a vector $\psi\ne0$, write
\[
 \supp_{\mathcal B}(\psi)
 =\{v\in\mathcal B:\langle v,\psi\rangle\ne0\},
 \qquad
 s_{\mathcal B}(\psi):=|\supp_{\mathcal B}(\psi)|.
\]
Thus $s_{\mathcal B}(\psi)$ is a support cardinality, not a norm.
The quantities $\mathcal S_p(\psi)$ and $T_s(p)$ are defined in
\eqref{eq:intro-Sp-Ts}.  Since only vanishing matters, all statements are
projective in $\psi$.

For two distinct standard MUBs, the transition matrix is equivalent to the
prime Fourier matrix under multiplication of rows and columns by nonzero
phases and under row and column permutations.  To see this, use
Lemma~\ref{lem:standard-mub-transitivity} below to send one of the two bases
to \(\mathcal B_\infty\); the other is then some \(\mathcal B_a\).  The
transition matrix from \(\mathcal B_\infty\) to \(\mathcal B_a\) has entries
\[
 p^{-1/2}\omega^{ax^2+bx}
 =\omega^{ax^2}\bigl(p^{-1/2}\omega^{bx}\bigr),
\]
so it differs from the prime Fourier matrix only by a diagonal row phase.
Tao's theorem is invariant under these operations and therefore gives
\begin{equation}
 s_{\mathcal B_j}(\psi)
 +
 s_{\mathcal B_k}(\psi)
 \ge p+1
 \qquad(j\ne k).
 \label{eq:pairwise-MUB-support}
\end{equation}
This is the two-basis input behind the earlier complete-MUB bound
\eqref{eq:intro-FW-bound}; the stronger results below exploit the simultaneous
quadratic structure of all $p$ finite bases.

\begin{lemma}[Transitivity of the standard MUB]
\label{lem:standard-mub-transitivity}
For every basis \(\mathcal B_j\) in the standard complete set
\[
 \{\mathcal B_\infty\}\cup\{\mathcal B_a:a\in\F_p\},
\]
there exists a unitary \(U\) that permutes the whole standard complete set of
MUBs, up to phases and permutations inside individual bases, and satisfies
\[
 U\mathcal B_j=\mathcal B_\infty.
\]
Consequently, the multiset of support sizes is preserved; only the basis
labels are permuted.
\end{lemma}

\begin{proof}
For \(c\in\F_p\), define the diagonal chirp
\[
 (C_cf)(x)=\omega^{cx^2}f(x).
\]
From the definition
\[
 \phi_{a,b}(x)=p^{-1/2}\omega^{ax^2+bx}
\]
one has
\[
 C_c\mathcal B_\infty=\mathcal B_\infty,
 \qquad
 C_c\mathcal B_a=\mathcal B_{a+c}.
\]
Let \(\mathsf F\) denote the unitary Fourier matrix
\[
 \mathsf F_{x,y}=p^{-1/2}\omega^{xy}.
\]
Then \(\mathsf F^*\mathcal B_0=\mathcal B_\infty\) and
\(\mathsf F^*\mathcal B_\infty=\mathcal B_0\), up to permutations of the basis
vectors.  If \(a\ne0\), then
\[
 (\mathsf F^*\phi_{a,b})(y)
 =
 \frac1p\sum_{x\in\F_p}\omega^{ax^2+(b-y)x}.
\]
Completing the square gives
\[
 ax^2+(b-y)x
 =
 a\left(x+\frac{b-y}{2a}\right)^2
 -
 \frac{(b-y)^2}{4a}.
\]
The quadratic Gauss sum
\[
 G(a):=\sum_{x\in\F_p}\omega^{ax^2}
\]
is nonzero.  Indeed, since \(p\) is odd, the change of variables
\[
 u=x-y,\qquad v=x+y
\]
is a bijection of \(\F_p^2\), and therefore
\[
 \begin{aligned}
 |G(a)|^2
 &=\sum_{x,y\in\F_p}\omega^{a(x^2-y^2)}
 =\sum_{u,v\in\F_p}\omega^{auv}\\
 &=\sum_{u\in\F_p}
   \left(\sum_{v\in\F_p}\omega^{auv}\right)
 =p.
 \end{aligned}
\]
Thus \(|G(a)|=\sqrt p\).  Hence
\[
 (\mathsf F^*\phi_{a,b})(y)
 =
 \frac{G(a)}p
 \omega^{-(b-y)^2/(4a)}
 =
 \zeta_{a,b}\,
 p^{-1/2}\omega^{-y^2/(4a)+(b/(2a))y},
\]
where \(\zeta_{a,b}\) is a nonzero scalar of modulus one.  Therefore
\[
 \mathsf F^*\mathcal B_a=\mathcal B_{-1/(4a)}
 \qquad(a\ne0),
\]
up to phases and a permutation of the second label.  Thus \(\mathsf F^*\) also
permutes the standard complete MUB.

If \(j=\infty\), take \(U=I\).  If \(j=a\in\F_p\), first apply
\(C_{-a}\), which sends \(\mathcal B_a\) to \(\mathcal B_0\), and then
apply \(\mathsf F^*\), which sends \(\mathcal B_0\) to
\(\mathcal B_\infty\).  The composition still permutes the complete
standard MUB.
\end{proof}

\section{Finite-parabola zero theorem}
\label{sec:parabolic-exact}

Throughout this section \(p\ge5\) is prime and
\(\omega=e^{2\pi i/p}\).  We first record the precise bridge between the finite-parabola sum and the complete MUB. If
\[
 \psi=\sum_{x\in R}c_xe_x
\]
has computational support \(R\subset\F_p\), then under the standard Hermitian
inner-product convention, for finite-basis labels \(A,\beta\in\F_p\),
\begin{equation}
 \langle \phi_{A,\beta},\psi\rangle
 =p^{-1/2}\sum_{x\in R}c_x\omega^{-A x^2-\beta x}
 =p^{-1/2}F_{R,c}(-A,-\beta).
 \label{eq:mub-parabola-label-map}
\end{equation}
where
\begin{equation}
 F_{R,c}(a,b)
 :=
 \sum_{x\in R}c_x\omega^{a x^2+b x},
 \qquad (a,b)\in\F_p^2.
 \label{eq:parabolic-sum-rigorous}
\end{equation}
Since \((A,\beta)\mapsto(-A,-\beta)\) is a bijection of \(\F_p^2\), the
number of zero coefficients among the \(p\) finite quadratic bases is exactly
\(\#Z(F_{R,c})\).
For such a sum write
\[
 Z(F_{R,c})
 :=\{(a,b)\in\F_p^2:F_{R,c}(a,b)=0\}.
\]
Whenever an auxiliary coefficient vector is allowed to have zero coordinates,
its \emph{actual support} means the set of indices on which its coefficients
are nonzero.  Thus every parabolic sum is always interpreted on its actual
support when a support-size theorem is applied.
The proof below establishes a quantitative upper bound for
$|Z(F_{R,c})|$.

\paragraph{Proof architecture.}
There are four steps.  First, a rank induction reduces arbitrary complex
coefficients either to smaller support or to a one-dimensional kernel defined
over \(K=\Q(\omega)\).  Second, cyclotomic localization converts zeros of \(F\)
into vanishing conditions for truncated jets over \(\F_p\).  Third, the first
nonzero jet is shown by a Vandermonde moment argument to occur by order
\(\lfloor k/2\rfloor\).  Finally, the parabolic jet recurrence reduces the
two-dimensional zero count to a multiplicity budget for the pencil
\(A(a)+bB(a)\).

\subsection{Algebraic preliminaries}

\begin{lemma}[Truncated cyclotomic model]
\label{lem:truncated-cyclotomic-model}
Let
\[
 K=\mathbb Q(\omega),\qquad
 \pi=1-\omega,
\]
and let
\[
 \mathcal O:=\mathbb Z[\omega]_{(\pi)}
\]
be the localization at the prime ideal generated by \(\pi\).  Then
\begin{equation}
 \mathcal O/(\pi^{p-1})
 \cong
 \F_p[X]/(X^{p-1}),
 \qquad
 \pi\longmapsto X.
 \label{eq:cyclotomic-truncation}
\end{equation}
Moreover the substitution
\begin{equation}
 X=1-e^T,
 \qquad
 e^T:=\sum_{j=0}^{p-2}\frac{T^j}{j!},
 \label{eq:formal-change-rigorous}
\end{equation}
defines an isomorphism
\[
 \F_p[X]/(X^{p-1})
 \cong
 \F_p[T]/(T^{p-1}).
\]
Under the composite isomorphism,
\begin{equation}
 \omega^h=(1-\pi)^h
 \longmapsto
 e^{hT}
 \pmod{T^{p-1}}
 \label{eq:omega-to-exp}
\end{equation}
for every \(h\in\F_p\).
\end{lemma}

\begin{proof}
The identity
\[
 \Phi_p(1)=\prod_{j=1}^{p-1}(1-\omega^j)=p
\]
shows that \(p\) is associated in \(\mathcal O\) to
\(\pi^{p-1}\): indeed
\[
 1-\omega^j
 =
 (1-\omega)(1+\omega+\cdots+\omega^{j-1}),
\]
and the second factor reduces to \(j\ne0\) modulo \(\pi\), hence is a
unit in \(\mathcal O\).  Thus
\[
 p=u\pi^{p-1}
\]
for some \(u\in\mathcal O^\times\).  In particular
\(\mathcal O/(\pi^{p-1})\) has characteristic \(p\).

Since \(\omega\equiv1\pmod\pi\), reduction modulo \(\pi\) sends
\(\mathbb Z[\omega]\) onto \(\mathbb Z/(p)\); hence the residue field is
\[
 \mathcal O/(\pi)\cong\F_p.
\]
The ring \(\mathbb Z[\omega]\) is the ring of integers of \(K\)
\cite[Ch.~2]{Washington}, hence is Dedekind, so its localization
\(\mathcal O\) at the nonzero prime \((\pi)\) is a discrete valuation ring
with uniformizer \(\pi\).
Successive reduction modulo powers of \(\pi\) shows that every class
modulo \(\pi^{p-1}\) can be written
\[
 a_0+a_1\pi+\cdots+a_{p-2}\pi^{p-2},
 \qquad a_j\in\F_p.
\]
Hence the homomorphism
\[
 \F_p[X]/(X^{p-1})
 \longrightarrow
 \mathcal O/(\pi^{p-1}),
 \qquad X\longmapsto\pi,
\]
is surjective.  Both rings have \(p^{p-1}\) elements: the left side has
\(\F_p\)-dimension \(p-1\), while the right side has a filtration
\[
 \mathcal O/(\pi^{p-1})
 \supset
 (\pi)/(\pi^{p-1})
 \supset\cdots\supset
 (\pi^{p-2})/(\pi^{p-1})
 \supset0
\]
whose \(p-1\) successive quotients are all isomorphic to the residue
field \(\F_p\).  The surjection is therefore an isomorphism, proving
\eqref{eq:cyclotomic-truncation}.

Because \(1,\ldots,p-2\) are invertible in \(\F_p\), the truncated
exponential in \eqref{eq:formal-change-rigorous} is well defined.
Furthermore
\[
 1-e^T=-T+O(T^2)
\]
has zero constant term and nonzero linear coefficient.  A substitution
of this form is an automorphism of the local Artinian algebra
\(\F_p[T]/(T^{p-1})\).  Explicitly, its inverse is given by the
truncated logarithm
\[
 T=\log(1-X)
 =
 -\sum_{j=1}^{p-2}\frac{X^j}{j}
 \pmod{X^{p-1}}.
\]
Finally, choose the standard integer representative
\(h\in\{0,1,\ldots,p-1\}\).  The formal identity
\[
 (e^T)^h=e^{hT}
\]
is first understood in \(\mathbb Q[[T]]\).  Reduce this identity modulo
\(T^{p-1}\).  Only the coefficients of degrees \(0,\ldots,p-2\) remain,
and every denominator occurring there divides a product of integers strictly
smaller than \(p\).  Those denominators are therefore units in
\(\mathbb Z_{(p)}\), so the truncated identity lies in
\(\mathbb Z_{(p)}[T]/(T^{p-1})\) and may be reduced coefficientwise modulo
\(p\).  Since \(1-X=e^T\) in the truncated algebra, this gives
\[
 (1-X)^h=e^{hT}\pmod{T^{p-1}},
\]
which proves \eqref{eq:omega-to-exp}.
\end{proof}

\begin{lemma}[Reduction from complex to cyclotomic coefficients]
\label{lem:complex-to-cyclotomic}
Fix \(R\subset\F_p\), \(|R|=k\ge2\), and a set
\(Z\subset\F_p^2\).  For \(y=(a,b)\in Z\), put
\[
 v_y:=
 \bigl(\omega^{a x^2+b x}\bigr)_{x\in R}
 \in K^R.
\]
Suppose there exists \(0\ne c\in\C^R\) such that
\[
 \sum_{x\in R}c_x(v_y)_x=0
 \qquad(y\in Z).
 \label{eq:zero-linear-system}
\]
Let \(r_Z\) be the rank over \(K\) of the matrix whose rows are the
\(v_y\).

Then exactly one of the following alternatives is available:

\begin{enumerate}
\item \(r_Z\le k-2\), in which case there exists
\(0\ne c'\in\C^R\) satisfying all equations
\eqref{eq:zero-linear-system} and having at most \(r_Z+1\) nonzero
coordinates (and hence at most \(k-1\));
\item \(r_Z=k-1\), in which case the common kernel is one-dimensional
and is generated by a vector in \(K^R\).  Hence \(c\) is a nonzero
complex scalar multiple of a vector in \(K^R\).
\end{enumerate}
\end{lemma}

\begin{proof}
Because the coefficient matrix has entries in \(K\), its rank is
unchanged after extending scalars from \(K\) to \(\C\).  Since the
nonzero vector \(c\) belongs to its kernel,
\[
 r_Z\le k-1.
\]

Assume first \(r_Z\le k-2\).  Let
\[
 W:=
 \left\{
 u\in\C^R:
 \sum_{x\in R}u_x(v_y)_x=0
 \text{ for every }y\in Z
 \right\}.
\]
Then
\[
 \dim_\C W=k-r_Z.
\]
Choose any set \(S\subset R\) of exactly \(k-r_Z-1\) coordinates and let
\[
 H_S:=\{u\in\C^R:u_x=0\text{ for every }x\in S\}.
\]
The subspace \(H_S\) has codimension \(k-r_Z-1\).  Therefore
\[
 \dim(W\cap H_S)
 \ge (k-r_Z)-(k-r_Z-1)=1.
\]
Choose \(0\ne c'\in W\cap H_S\).  It satisfies every equation
\eqref{eq:zero-linear-system} and can be nonzero on at most
\[
 k-|S|=r_Z+1
\]
coordinates.  This proves the first alternative.

Now assume \(r_Z=k-1\).  Choose \(k-1\) linearly independent rows.  The
resulting \((k-1)\times k\) matrix \(M\) has a one-dimensional kernel.
The vector of signed maximal minors
\[
 w_j=(-1)^{j-1}\det M^{(j)},
\]
where \(M^{(j)}\) is obtained by deleting the \(j\)-th column, is a
nonzero vector in \(K^R\) and satisfies \(Mw=0\).  Therefore the kernel
over \(K\) is \(Kw\), while after extending scalars the kernel over
\(\C\) is \(\C w\).  Hence \(c=\lambda w\) for some
\(\lambda\in\C^\times\).
\end{proof}

\subsection{Cyclotomic jets}

\begin{lemma}[First nonzero jet]
\label{lem:first-nonzero-jet}
Let
\[
 R\subset\F_p,\qquad 2\le k:=|R|\le p-1,
\]
and let \(c_x\in K^\times\) for every \(x\in R\).  After multiplication by a common
nonzero scalar, assume
\begin{equation}
 c_x\in\mathcal O
 \quad(x\in R),
 \qquad
 \min_{x\in R}v_\pi(c_x)=0.
 \label{eq:primitive-c-rigorous}
\end{equation}
Map each \(c_x\) through Lemma~\ref{lem:truncated-cyclotomic-model} and
write its image as
\[
 C_x(T)\in\F_p[T]/(T^{p-1}).
\]
Define
\begin{equation}
 \mathcal G(T;a,b)
 :=
 \sum_{x\in R}
 C_x(T)e^{T(a x^2+b x)}
 =
 \sum_{j=0}^{p-2}H_j(a,b)T^j
 \pmod{T^{p-1}}.
 \label{eq:G-rigorous}
\end{equation}
Then:

\begin{enumerate}
\item \(H_j\in\F_p[a,b]\) has total degree at most \(j\);
\item if \(n\) is the least index for which \(H_n\) is not the zero
polynomial, then
\begin{equation}
 n\le\left\lfloor\frac{k}{2}\right\rfloor.
 \label{eq:n-rigorous}
\end{equation}
\end{enumerate}
\end{lemma}

\begin{proof}
Write
\[
 C_x(T)=\sum_{\ell=0}^{p-2}c_{x,\ell}T^\ell.
\]
The coefficient of \(T^j\) in
\[
 C_x(T)e^{T(a x^2+b x)}
\]
is
\[
 \sum_{\ell=0}^{j}
 c_{x,\ell}
 \frac{(a x^2+b x)^{j-\ell}}{(j-\ell)!}.
\]
Hence \(H_j\) has total degree at most \(j\), proving the first assertion.

Put
\[
 r:=\left\lfloor\frac{k}{2}\right\rfloor.
\]
Suppose, for contradiction, that
\[
 H_0=H_1=\cdots=H_r=0
\]
as polynomials.  For each \(0\le j\le r\), the homogeneous component of
total degree \(j\) of \(H_j\) comes only from the terms with
\(\ell=0\), and therefore equals
\begin{equation}
 \frac1{j!}
 \sum_{x\in R}
 c_{x,0}(a x^2+b x)^j.
 \label{eq:top-homogeneous-rigorous}
\end{equation}
Expanding,
\[
 (a x^2+b x)^j
 =
 \sum_{t=0}^{j}
 \binom jt
 a^t b^{j-t}x^{j+t}.
\]
Since \(j<p\), all coefficients \(j!\) and \(\binom jt\) occurring here
are nonzero in \(\F_p\).  Vanishing of
\eqref{eq:top-homogeneous-rigorous} therefore implies
\begin{equation}
 \sum_{x\in R}c_{x,0}x^d=0
 \qquad
 (j\le d\le2j).
 \label{eq:moment-window}
\end{equation}
As \(j\) runs from \(0\) to \(r\), the intervals
\([j,2j]\) cover every integer
\[
 0,1,\ldots,k-1.
\]
Thus
\[
 \sum_{x\in R}c_{x,0}x^d=0
 \qquad(0\le d\le k-1).
\]
Fix an ordering
\[
 R=\{x_1,\ldots,x_k\}.
\]
The Vandermonde matrix
\[
 V=(x_j^d)_{\substack{0\le d\le k-1\\1\le j\le k}}
\]
has determinant
\[
 \det V=\prod_{1\le i<j\le k}(x_j-x_i)\ne0
\]
in \(\F_p\), because the elements of \(R\) are distinct.  Hence
\(c_{x,0}=0\) for every \(x\in R\).

On the other hand, the primitivity condition
\eqref{eq:primitive-c-rigorous} means that at least one \(c_x\) is a unit
of \(\mathcal O\).  Its image modulo \(\pi\), which is precisely
\(c_{x,0}\), is nonzero.  This contradiction proves
\eqref{eq:n-rigorous}.
\end{proof}

\begin{lemma}[Parabolic jet recurrence]
\label{lem:schrodinger-jet}
Retain the notation of Lemma~\ref{lem:first-nonzero-jet}, let \(n\) be
the first nonzero jet order, and put
\[
 Q_j(a,b):=H_{n+j}(a,b)
\]
whenever \(n+j\le p-2\).  Then
\begin{equation}
 \partial_b^2Q_0=0,
 \qquad
 \partial_b^2Q_j=\partial_aQ_{j-1}
 \quad(j\ge1).
 \label{eq:schrodinger-recurrence-rigorous}
\end{equation}
Consequently
\begin{equation}
 Q_0(a,b)=A(a)+bB(a)
 \label{eq:leading-pencil-rigorous}
\end{equation}
for polynomials \(A,B\in\F_p[a]\) satisfying
\begin{equation}
 \deg A\le n,
 \qquad
 \deg B\le n-1.
 \label{eq:AB-degree-rigorous}
\end{equation}
Moreover, whenever \(n+j\le p-2\),
\begin{equation}
 \deg_bQ_j\le2j+1.
 \label{eq:b-degree-rigorous}
\end{equation}
\end{lemma}

\begin{proof}
In the truncated polynomial ring,
\[
 \partial_b e^{T(a x^2+b x)}
 =
 Tx\,e^{T(a x^2+b x)}
\]
and
\[
 \partial_a e^{T(a x^2+b x)}
 =
 Tx^2\,e^{T(a x^2+b x)}.
\]
Therefore
\begin{equation}
 \partial_b^2\mathcal G
 =
 T\,\partial_a\mathcal G
 \pmod{T^{p-1}}.
 \label{eq:formal-schrodinger-rigorous}
\end{equation}
Comparing coefficients of \(T^n\) gives
\(\partial_b^2H_n=0\).  Comparing coefficients of
\(T^{n+j}\), \(j\ge1\), gives
\[
 \partial_b^2H_{n+j}
 =
 \partial_aH_{n+j-1}.
\]
This is \eqref{eq:schrodinger-recurrence-rigorous}.

Since \(\deg Q_0\le n<p\), write
\[
 Q_0(a,b)=\sum_{d=0}^{n}q_d(a)b^d.
\]
The equation \(\partial_b^2Q_0=0\) gives
\[
 d(d-1)q_d(a)=0
 \qquad(d\ge2).
\]
Because \(2\le d<p\), the scalar \(d(d-1)\) is nonzero in \(\F_p\).
Thus \(q_d=0\) for every \(d\ge2\), proving
\eqref{eq:leading-pencil-rigorous}.  Since the total degree of \(Q_0\)
is at most \(n\), the degree bounds
\eqref{eq:AB-degree-rigorous} follow.

It remains to prove \eqref{eq:b-degree-rigorous}.  We use induction on
\(j\).  The case \(j=0\) has just been proved.  Suppose
\(\deg_bQ_{j-1}\le2j-1\).  Then
\[
 \deg_b(\partial_aQ_{j-1})\le2j-1.
\]
Also \(n+j\le p-2\), hence every power of \(b\) occurring in \(Q_j\)
has exponent \(<p\).  If \(Q_j\) contained a nonzero term
\(q_d(a)b^d\) with \(d\ge2j+2\), then the term
\[
 d(d-1)q_d(a)b^{d-2}
\]
would survive in \(\partial_b^2Q_j\), because \(d<p\), and would have
\(b\)-degree at least \(2j\).  This contradicts
\[
 \partial_b^2Q_j=\partial_aQ_{j-1}.
\]
Therefore \(\deg_bQ_j\le2j+1\).
\end{proof}

\begin{lemma}[Zeros of the parabolic sum vanish in every available jet]
\label{lem:exact-zero-all-jets}
Under the hypotheses of Lemma~\ref{lem:first-nonzero-jet}, define
\[
 F(a,b)=\sum_{x\in R}c_x\omega^{a x^2+b x}.
\]
If
\[
 F(\alpha,\beta)=0
\]
in \(K\), then
\begin{equation}
 H_j(\alpha,\beta)=0
 \qquad(0\le j\le p-2).
 \label{eq:exact-zero-all-jets}
\end{equation}
Equivalently, if \(Q_j\) is defined, then
\[
 Q_j(\alpha,\beta)=0.
\]
\end{lemma}

\begin{proof}
By \eqref{eq:primitive-c-rigorous}, every summand of \(F(\alpha,\beta)\)
lies in \(\mathcal O\).  The exact equality
\[
 F(\alpha,\beta)=0
\]
therefore remains zero after reduction modulo \(\pi^{p-1}\).
Under the isomorphism of Lemma~\ref{lem:truncated-cyclotomic-model}, this
reduced element is exactly
\[
 \mathcal G(T;\alpha,\beta)
 =
 \sum_{j=0}^{p-2}H_j(\alpha,\beta)T^j
 \in\F_p[T]/(T^{p-1}).
\]
The residue classes \(1,T,\ldots,T^{p-2}\) are linearly independent
over \(\F_p\).  Hence every coefficient is zero, which is precisely
\eqref{eq:exact-zero-all-jets}.
\end{proof}

The implication in Lemma~\ref{lem:exact-zero-all-jets} is deliberately one-way: an exact cyclotomic zero forces every available finite-field jet to vanish, but simultaneous jet vanishing is not used as a converse criterion for an exact zero. In the later extremizer construction, exact zeros are established directly from the explicit constraints, and the zero theorem is then used only to show that no additional zeros can occur.

\subsection{The zero bound}

\begin{lemma}[Cyclotomic parabolic zero bound]
\label{lem:cyclotomic-zero-rigorous}
Let
\[
 R\subset\F_p,\qquad 2\le k:=|R|\le p-1,
\]
and let \(c_x\in K^\times\) for every \(x\in R\).  Define
\[
 F(a,b)=\sum_{x\in R}c_x\omega^{a x^2+b x}.
\]
Then
\begin{equation}
 \#Z(F)
 :=
 \#\{(a,b)\in\F_p^2:F(a,b)=0\}
 \le
 p+2\left\lfloor\frac{k}{2}\right\rfloor-2.
 \label{eq:cyclotomic-zero-rigorous}
\end{equation}
\end{lemma}

\begin{proof}
Scale \(c\) so that \eqref{eq:primitive-c-rigorous} holds, and let
\(n\) be the first nonzero jet order.  By
Lemma~\ref{lem:first-nonzero-jet},
\begin{equation}
 n\le r:=\left\lfloor\frac{k}{2}\right\rfloor.
 \label{eq:n-vs-r}
\end{equation}
Write
\[
 Q_0(a,b)=A(a)+bB(a)
\]
as in Lemma~\ref{lem:schrodinger-jet}.

\medskip
\noindent
\textbf{Case 1: \(B\not\equiv0\).}
Let
\[
 D=\gcd(A,B),
\]
chosen monic.  We use the convention
\[
 \operatorname{ord}_\alpha 0=+\infty.
\]
Fix \(\alpha\in\F_p\).

If \(B(\alpha)\ne0\), then
\[
 A(\alpha)+bB(\alpha)=0
\]
has exactly one solution \(b\in\F_p\).  Since every zero of \(F\) must
annihilate \(Q_0\) by Lemma~\ref{lem:exact-zero-all-jets}, the vertical
fiber \(a=\alpha\) contains at most one zero of \(F\).

If \(B(\alpha)=0\) but \(A(\alpha)\ne0\), then the fiber contains no
zeros of \(F\).

It remains to consider a common root
\(\alpha\) of \(A\) and \(B\).  Put
\[
 m_\alpha
 :=
 \min\{\operatorname{ord}_\alpha A,
        \operatorname{ord}_\alpha B\}.
\]
Because \(B\not\equiv0\),
\[
 1\le m_\alpha\le\deg B\le n-1.
\]
Therefore
\[
 n+m_\alpha
 \le2n-1
 \le2r-1
 \le p-2,
\]
so \(Q_{m_\alpha}\) is defined.

Iterating the recurrence
\eqref{eq:schrodinger-recurrence-rigorous} \(m_\alpha\) times gives
\begin{equation}
 \partial_b^{2m_\alpha}Q_{m_\alpha}
 =
 \partial_a^{m_\alpha}Q_0.
 \label{eq:iterated-recurrence-rigorous}
\end{equation}
Indeed, the statement follows by induction:
\[
 \partial_b^{2s}Q_{m_\alpha}
 =
 \partial_a^sQ_{m_\alpha-s}
 \qquad(0\le s\le m_\alpha).
\]

At \(a=\alpha\),
\[
 \partial_a^{m_\alpha}Q_0(\alpha,b)
 =
 A^{(m_\alpha)}(\alpha)
 +
 bB^{(m_\alpha)}(\alpha).
 \label{eq:nonzero-linear-derivative}
\]
By definition of \(m_\alpha\), both \(A\) and \(B\) vanish to order at
least \(m_\alpha\), and at least one vanishes to exactly that order.
Since \(m_\alpha<p\), the factor \(m_\alpha!\) is nonzero in \(\F_p\).
Hence the polynomial in
\eqref{eq:nonzero-linear-derivative} is not identically zero.

It follows from \eqref{eq:iterated-recurrence-rigorous} that
\(Q_{m_\alpha}(\alpha,b)\) is not the zero polynomial.  By
\eqref{eq:b-degree-rigorous},
\[
 \deg_bQ_{m_\alpha}(\alpha,b)\le2m_\alpha+1<p.
\]
Every zero of \(F\) on this fiber is a root of this polynomial, by
Lemma~\ref{lem:exact-zero-all-jets}.  Hence the fiber contains at most
\begin{equation}
 2m_\alpha+1
 \label{eq:fiber-bound-common-root}
\end{equation}
zeros of \(F\).

Let \(Z_{\F_p}(B)\) be the set of roots of \(B\) in \(\F_p\), and let
\(Z_{\F_p}(D)\) be the common roots of \(A\) and \(B\) in \(\F_p\).
Adding the fiber bounds gives
\[
 \begin{aligned}
 \#Z(F)
 &\le
 p-|Z_{\F_p}(B)|
 +\sum_{\alpha\in Z_{\F_p}(D)}(2m_\alpha+1)\\
 &=
 p
 +2\sum_{\alpha\in Z_{\F_p}(D)}m_\alpha
 +|Z_{\F_p}(D)|-|Z_{\F_p}(B)|\\
 &\le
 p+2\sum_{\alpha\in Z_{\F_p}(D)}m_\alpha\\
 &\le
 p+2\deg D\\
 &\le
 p+2\deg B\\
 &\le
 p+2n-2.
 \end{aligned}
 \label{eq:count-case-one-rigorous}
\]

\medskip
\noindent
\textbf{Case 2: \(B\equiv0\).}
Then \(Q_0=A(a)\) with \(A\ne0\).  By
Lemma~\ref{lem:exact-zero-all-jets}, a zero of \(F\) can occur only on a
fiber \(a=\alpha\) for which \(A(\alpha)=0\).

Let \(\alpha\in\F_p\) be a root of multiplicity
\[
 m_\alpha:=\operatorname{ord}_\alpha A.
\]
If
\[
 n+m_\alpha\le p-2,
 \label{eq:safe-root-order}
\]
then the same recurrence argument as in Case 1 gives
\[
 \partial_b^{2m_\alpha}Q_{m_\alpha}(\alpha,b)
 =
 A^{(m_\alpha)}(\alpha)\ne0.
\]
Thus \(Q_{m_\alpha}(\alpha,b)\) is a nonzero polynomial of degree at
most \(2m_\alpha+1\), and the fiber contains at most
\(2m_\alpha+1\) zeros of \(F\).

Suppose first that \eqref{eq:safe-root-order} holds for every
\(\F_p\)-root of \(A\).  Then
\[
 \begin{aligned}
 \#Z(F)
 &\le
 \sum_{\alpha\in Z_{\F_p}(A)}(2m_\alpha+1)\\
 &=
 2\sum_{\alpha\in Z_{\F_p}(A)}m_\alpha
 +|Z_{\F_p}(A)|\\
 &\le
 2\deg A+\deg A\\
 &\le3n.
 \end{aligned}
\]
For \(n\ge1\),
\[
 3n\le p+2n-2
\]
because \(n\le r\le(p-1)/2\le p-2\).  If \(n=0\), then \(A\) is a
nonzero constant and \(\#Z(F)=0\).

It remains to treat a root for which
\[
 n+m_\alpha\ge p-1.
\]
Since
\[
 m_\alpha\le\deg A\le n
 \quad\text{and}\quad
 n\le r\le\frac{p-1}{2},
\]
all inequalities must be equalities:
\[
 k=p-1,
 \qquad
 n=m_\alpha=\frac{p-1}{2}.
 \label{eq:boundary-equalities}
\]
In particular \(\deg A=n\), and \(A\) has the single root
\(\alpha\) with multiplicity \(n\).  Hence every zero of \(F\)
lies on the single fiber \(a=\alpha\).

On that fiber,
\[
 b\longmapsto F(\alpha,b)
 =
 \sum_{x\in R}
 \bigl(c_x\omega^{\alpha x^2}\bigr)\omega^{bx}
\]
is the ordinary Fourier transform on \(\F_p\) of the nonzero function
\[
 x\longmapsto c_x\omega^{\alpha x^2}\mathbf 1_R(x),
\]
whose support is \(R\), of size \(p-1\).  Tao's
prime-dimensional uncertainty principle \cite{Tao} gives
\[
 |\supp F(\alpha,\cdot)|
 \ge
 p+1-(p-1)=2.
\]
Thus this fiber contains at most \(p-2\) zeros.  Since all other fibers
contain none,
\[
 \#Z(F)\le p-2.
\]
This is stronger than the desired bound.

Combining both cases,
\[
 \#Z(F)\le p+2n-2
 \le
 p+2\left\lfloor\frac{k}{2}\right\rfloor-2,
\]
which proves \eqref{eq:cyclotomic-zero-rigorous}.
\end{proof}

\begin{theorem}[Parabolic zero theorem]
\label{thm:parabolic-zero-rigorous}
Let
\[
 R\subset\F_p,\qquad 1\le k:=|R|\le p-1,
\]
and let \(c_x\in\C^\times\) for every \(x\in R\).  Define
\[
 F(a,b)=\sum_{x\in R}c_x\omega^{a x^2+b x}.
\]
If \(k=1\), then \(F\) has no zeros.  If \(k\ge2\), then
\begin{equation}
 \#Z(F)
 \le
 p+2\left\lfloor\frac{k}{2}\right\rfloor-2.
 \label{eq:parabolic-zero-rigorous}
\end{equation}
Equivalently,
\[
 \#Z(F)\le
 \begin{cases}
 p+k-2,&k\text{ even},\\
 p+k-3,&k\ge3\text{ odd}.
 \end{cases}
\]
\end{theorem}

\begin{proof}
We use induction on \(k\).

For \(k=1\),
\[
 F(a,b)=c_x\omega^{a x^2+b x}
\]
for the unique \(x\in R\), so \(F(a,b)\ne0\) everywhere.

Assume \(k\ge2\), and assume the theorem has been proved for every
strictly smaller support size.  Let
\[
 Z:=Z(F).
\]
For \(y=(a,b)\in Z\), define
\[
 v_y=
 \bigl(\omega^{a x^2+b x}\bigr)_{x\in R}.
\]
The coefficient vector \(c\) satisfies the linear system
\[
 \sum_{x\in R}c_x(v_y)_x=0
 \qquad(y\in Z).
\]
Apply Lemma~\ref{lem:complex-to-cyclotomic}.

If the rank \(r_Z\le k-2\), that lemma provides a nonzero vector \(c'\)
satisfying all the same zero equations and having at most \(k-1\) nonzero
coordinates.  Let
\[
 R':=\{x\in R:c'_x\ne0\},
\]
so \(1\le |R'|<k\), and let \(F'\) be the corresponding parabolic sum
written on its actual support \(R'\).
Then
\[
 Z(F)\subseteq Z(F').
\]
If \(|R'|=1\), then \(F'\) has no zeros, so the required estimate is
immediate.  If \(|R'|\ge2\), the induction hypothesis gives
\[
 \begin{aligned}
 |Z(F)|
 &\le |Z(F')|\\
 &\le
 p+2\left\lfloor\frac{|R'|}{2}\right\rfloor-2\\
 &\le
 p+2\left\lfloor\frac{k}{2}\right\rfloor-2.
 \end{aligned}
\]

If \(r_Z=k-1\), Lemma~\ref{lem:complex-to-cyclotomic} shows that
\(c\) is a nonzero complex scalar multiple of a vector in \(K^R\).
Multiplication by a nonzero scalar does not change the zero set.
Therefore Lemma~\ref{lem:cyclotomic-zero-rigorous} applies and gives
the required estimate.

These two cases exhaust all possibilities because the existence of the
nonzero kernel vector \(c\) implies \(r_Z\le k-1\).
\end{proof}

\section{Exact complete-MUB uncertainty}
\label{sec:consequences}

The parabolic zero theorem now yields the sharp total-support bound.

\begin{theorem}[Exact complete-MUB uncertainty]
\label{thm:exact-mub-rigorous}
For every prime \(p\ge5\) and every nonzero
\(\psi\in\C^p\),
\begin{equation}
 \mathcal S_p(\psi)\ge p^2-p+2.
 \label{eq:exact-mub-bound-rigorous}
\end{equation}
The bound is attained.  Hence the exact minimum total support is
\begin{equation}
 T_s(p)=p^2-p+2.
 \label{eq:Ts-rigorous}
\end{equation}
\end{theorem}

\begin{proof}
Among the \(p+1\) standard MUBs, choose one in which \(\psi\) has
minimum support, and denote that support size by \(k\).
By Lemma~\ref{lem:standard-mub-transitivity}, after applying a unitary
which merely permutes the standard MUBs, we may assume that this basis is
the computational basis.  Thus
\[
 R:=\supp_{\mathcal B_\infty}(\psi)
\]
has size \(k\).

If \(k=p\), minimality implies that every one of the \(p+1\) support
sizes equals \(p\), so
\[
 \mathcal S_p(\psi)=p(p+1)>p^2-p+2.
\]

If \(k=1\), then \(\psi\) is a scalar multiple of one computational
basis vector.  Mutual unbiasedness implies that all \(p\) coordinates
in each of the other \(p\) bases are nonzero.  Hence
\[
 \mathcal S_p(\psi)=1+p^2>p^2-p+2.
\]

Assume now
\[
 2\le k\le p-1.
\]
Write
\[
 \psi=\sum_{x\in R}c_x e_x,
 \qquad c_x\ne0.
\]
For \(a,b\in\F_p\),
\[
 \langle\phi_{a,b},\psi\rangle
 =
 p^{-1/2}
 \sum_{x\in R}
 c_x\omega^{-a x^2-b x}.
\]
Replacing \((a,b)\) by \((-a,-b)\), which is a bijection of
\(\F_p^2\), shows that the total number \(N\) of zero outcomes among the
\(p\) finite quadratic bases is
\[
 N
 =
 \#\left\{
 (a,b)\in\F_p^2:
 \sum_{x\in R}c_x\omega^{a x^2+b x}=0
 \right\}.
\]
By Theorem~\ref{thm:parabolic-zero-rigorous},
\[
 N
 \le
 p+2\left\lfloor\frac{k}{2}\right\rfloor-2.
\]
The computational basis contributes support \(k\), while the \(p\)
finite bases contain \(p^2-N\) nonzero outcomes.  Therefore
\[
 \begin{aligned}
 \mathcal S_p(\psi)
 &=k+p^2-N\\
 &\ge
 k+p^2-
 \left(
 p+2\left\lfloor\frac{k}{2}\right\rfloor-2
 \right)\\
 &=
 p^2-p+
 k-2\left\lfloor\frac{k}{2}\right\rfloor+2\\
 &\ge p^2-p+2.
 \end{aligned}
\]
For odd \(k\), the displayed estimate is in fact
\(p^2-p+3\).

It remains to show that \eqref{eq:exact-mub-bound-rigorous} is attained.
Choose distinct \(r,s\in\F_p\) and \(t\in\F_p\), and put
\[
 \psi=e_r-\omega^t e_s.
\]
Its computational support is \(2\).  In the basis \(\mathcal B_a\),
\[
 \langle\phi_{a,b},\psi\rangle
 =
 p^{-1/2}
 \left(
 \omega^{-ar^2-br}
 -
 \omega^t\omega^{-as^2-bs}
 \right).
\]
This coefficient vanishes if and only if
\[
 -ar^2-br
 \equiv
 t-as^2-bs
 \pmod p,
\]
equivalently,
\begin{equation}
 b(r-s)
 =
 -a(r^2-s^2)-t.
 \label{eq:unique-zero-each-basis}
\end{equation}
Since \(r-s\ne0\), for every fixed \(a\in\F_p\) equation
\eqref{eq:unique-zero-each-basis} has exactly one solution
\(b\in\F_p\).  Thus each of the \(p\) finite quadratic bases has exactly
one zero coefficient and therefore support \(p-1\).  Hence
\[
 \mathcal S_p(\psi)
 =
 2+p(p-1)
 =
 p^2-p+2.
\]
This proves both sharpness and \eqref{eq:Ts-rigorous}.
\end{proof}

\section{Rigidity of equality}
\label{sec:extremizer-structure}

To avoid a sign ambiguity in the equality theory, we distinguish the two
coordinate systems from this point onward.  Lowercase \((a,b)\) always denotes
the coordinates of the parabolic sum \(F(a,b)\).  For the corresponding MUB labels, we write \((A,\beta)\), where
\(\phi_{A,\beta}\in\mathcal B_A\).  By
\eqref{eq:mub-parabola-label-map},
\begin{equation}
 A=-a,\qquad \beta=-b.
 \label{eq:F-to-physical-labels}
\end{equation}

The sharp constant in Theorem~\ref{thm:exact-mub-rigorous} forces equality
at every stage of the parabolic zero argument.  We record the resulting
rigidity before turning to the full classification.

\begin{theorem}[Rigidity of extremizers]
\label{thm:extremizer-structure}
Let \(p\ge5\) be prime, and suppose that
\[
 \mathcal S_p(\psi)=p^2-p+2.
\]
Choose a standard MUB in which \(\psi\) has minimum support and, using
Lemma~\ref{lem:standard-mub-transitivity}, normalize it to the
computational basis.  Put
\[
 k:=|\supp_{\mathcal B_\infty}(\psi)|.
\]
Then the following assertions hold.

\begin{enumerate}
\item The integer \(k\) is even.  Write \(k=2r\).

\item If \(k=2\), then every finite quadratic basis contains exactly one
zero coefficient.  Moreover, if the computational support is
\(\{u,v\}\), then, up to a nonzero scalar,
\[
 \psi=e_u-\omega^t e_v
\]
for some \(t\in\F_p\).

\item Suppose \(k\ge4\).  Then, after multiplying \(\psi\) by a nonzero
complex scalar, its nonzero computational coordinates belong to
\(K=\mathbb Q(\omega)\).  In the notation of
Lemma~\ref{lem:schrodinger-jet}, the first nonzero jet has order
\[
 n=r,
\]
and its leading pencil factors as
\begin{equation}
 Q_0(a,b)=B(a)\bigl(b+L(a)\bigr),
 \label{eq:extremizer-leading-factorization}
\end{equation}
where
\[
 \deg B=r-1,
 \qquad
 \deg L\le1,
\]
and \(B\) splits completely over \(\F_p\).

If the distinct roots of \(B\) are
\(\alpha_1,\ldots,\alpha_t\) with multiplicities
\(m_1,\ldots,m_t\), then
\begin{equation}
 m_1+\cdots+m_t=r-1.
 \label{eq:extremizer-partition}
\end{equation}
The \(F\)-fiber \(a=\alpha_i\) has exactly
\begin{equation}
 2m_i+1
 \label{eq:extremizer-exceptional-zeros}
\end{equation}
zeros.  Equivalently, putting \(A_i=-\alpha_i\), the finite MUB
\(\mathcal B_{A_i}\) has exactly \(2m_i+1\) zero coefficients, while every
other finite MUB has exactly one zero coefficient.

\item Since the computational basis has minimum support, each
\(m_i\) satisfies
\begin{equation}
 2m_i+1\le p-k.
 \label{eq:extremizer-minimum-support-constraint}
\end{equation}
Consequently, if \(k\ge4\), then necessarily \(k\le p-3\).
\end{enumerate}

Thus, up to permutation of the \(p+1\) basis labels, the zero-count
profile of an extremizer with minimum support \(k=2r\ge4\) has the form
\begin{equation}
 \bigl(
 p-k,\,
 2m_1+1,\ldots,2m_t+1,\,
 \underbrace{1,\ldots,1}_{p-t\text{ entries}}
 \bigr),
 \label{eq:extremizer-zero-profile}
\end{equation}
where \((m_1,\ldots,m_t)\) is a partition of \(r-1\) satisfying
\eqref{eq:extremizer-minimum-support-constraint}.
\end{theorem}

\begin{proof}
Choose the minimum-support basis and normalize it to
\(\mathcal B_\infty\), exactly as in the proof of
Theorem~\ref{thm:exact-mub-rigorous}.  Let
\[
 R=\supp_{\mathcal B_\infty}(\psi),
 \qquad |R|=k,
\]
and let \(N\) denote the total number of zero coefficients among the
\(p\) finite quadratic bases.  Then
\begin{equation}
 \mathcal S_p(\psi)=k+p^2-N.
 \label{eq:extremizer-S-vs-N}
\end{equation}
Since \(\mathcal S_p(\psi)=p^2-p+2\),
\begin{equation}
 N=p+k-2.
 \label{eq:extremizer-N}
\end{equation}

Theorem~\ref{thm:parabolic-zero-rigorous} gives
\[
 N\le p+2\left\lfloor\frac{k}{2}\right\rfloor-2.
\]
Comparing with \eqref{eq:extremizer-N} yields
\[
 k\le2\left\lfloor\frac{k}{2}\right\rfloor.
\]
The reverse inequality is automatic, hence equality holds and \(k\) is
even.  Write \(k=2r\).  This proves (1).

Assume first \(k=2\), and write
\[
 R=\{u,v\},
 \qquad
 \psi=c_u e_u+c_v e_v,
 \qquad c_uc_v\ne0.
\]
Equation \eqref{eq:extremizer-N} gives \(N=p\).  On the other hand,
for every fixed \(a\in\F_p\), the function
\[
 b\longmapsto
 c_u\omega^{a u^2+bu}+c_v\omega^{a v^2+bv}
\]
has at most one zero, because after division by the nonzero first term
a zero is equivalent to one equation in the nontrivial character
\(b\mapsto\omega^{b(v-u)}\).  Since there are \(p\) finite bases and
there are altogether \(p\) zeros, each finite basis has exactly one.
In particular, for \(a=0\) there is \(b_0\in\F_p\) such that
\[
 c_u\omega^{b_0u}+c_v\omega^{b_0v}=0.
\]
Therefore
\[
 \frac{c_v}{c_u}=-\omega^{b_0(u-v)}.
\]
After multiplication by \(c_u^{-1}\), this is the asserted form
\(e_u-\omega^t e_v\).  This proves (2).

Now suppose \(k=2r\ge4\).  We revisit the proof of
Theorem~\ref{thm:parabolic-zero-rigorous} and show that every inequality
used there must be an equality.  Equality excludes the lower-rank branch,
as the following strict-gap argument makes explicit.

Let \(Z\) be the zero set of the corresponding parabolic sum, so
\(|Z|=N=p+2r-2\).  If the rank of the zero-constraint system were at most \(k-2\), the proof
of Lemma~\ref{lem:complex-to-cyclotomic} would produce a nonzero
parabolic sum \(F'\) on an actual support \(R'\) with
\(|R'|\le k-1\), while preserving every zero in \(Z\).  If
\(|R'|=1\), then \(F'\) has no zeros, impossible.  If \(|R'|\ge2\),
Theorem~\ref{thm:parabolic-zero-rigorous} would give
\[
 |Z|
 \le
 p+2\left\lfloor\frac{|R'|}{2}\right\rfloor-2
 \le
 p+2(r-1)-2
 =p+2r-4,
\]
contrary to \(|Z|=p+2r-2\).  Hence the rank of the zero-constraint system is exactly
\(k-1\).  Lemma~\ref{lem:complex-to-cyclotomic} therefore implies that,
after a nonzero scalar multiplication, the coefficient vector belongs
to \(K^R\).

We may consequently use the cyclotomic proof of
Lemma~\ref{lem:cyclotomic-zero-rigorous}.  Let \(n\) be the first
nonzero jet order.  That proof gives
\begin{equation}
 N\le p+2n-2
 \le p+2r-2.
 \label{eq:extremizer-chain-n}
\end{equation}
Since the two endpoints are equal to \(N\), both inequalities in
\eqref{eq:extremizer-chain-n} are equalities.  Thus
\begin{equation}
 n=r.
 \label{eq:extremizer-n-r}
\end{equation}

Write
\[
 Q_0(a,b)=A(a)+bB(a),
 \qquad
 \deg A\le r,
 \quad
 \deg B\le r-1.
\]
We first show that \(B\not\equiv0\).  If \(B\equiv0\), Case 2 in the
proof of Lemma~\ref{lem:cyclotomic-zero-rigorous} gives, away from the
single separately treated boundary configuration, the stronger bound
\[
 N\le3n=3r.
\]
But
\[
 3r<p+2r-2,
\]
because \(r\le(p-1)/2\) and \(p\ge5\).  In the boundary configuration
that proof gives the still stronger estimate \(N\le p-2\).  Both
contradict \(N=p+2r-2\).  Hence \(B\not\equiv0\).

Let \(D=\gcd(A,B)\), chosen monic.  Case 1 of the same proof gives the
chain
\begin{equation}
 \begin{aligned}
 N
 &\le
 p+2\sum_{\alpha\in Z_{\F_p}(D)}m_\alpha\\
 &\le p+2\deg D\\
 &\le p+2\deg B\\
 &\le p+2r-2,
 \end{aligned}
 \label{eq:extremizer-equality-chain}
\end{equation}
where
\[
 m_\alpha=
 \min\{\operatorname{ord}_\alpha A,
        \operatorname{ord}_\alpha B\}.
\]
Again the first and last quantities in
\eqref{eq:extremizer-equality-chain} are both equal to \(N\), so every
inequality in that chain is an equality.  In particular,
\begin{equation}
 \deg B=r-1,
 \qquad
 \deg D=\deg B.
 \label{eq:extremizer-degrees}
\end{equation}
Since \(D\mid B\) and the two polynomials have the same degree,
\(D\) differs from \(B\) only by a nonzero scalar.  Thus \(B\mid A\).
Because \(\deg A\le r\) and \(\deg B=r-1\), there is a polynomial
\(L\) of degree at most one such that
\[
 A=BL.
\]
This proves the factorization
\eqref{eq:extremizer-leading-factorization}.

Equality in
\[
 \sum_{\alpha\in Z_{\F_p}(D)}m_\alpha\le\deg D
\]
forces every root of \(D\), with its full multiplicity, to lie in
\(\F_p\).  Since \(D\) is a scalar multiple of \(B\), the polynomial
\(B\) splits completely over \(\F_p\).  If its distinct roots are
\(\alpha_1,\ldots,\alpha_t\) with multiplicities
\(m_1,\ldots,m_t\), then
\[
 m_1+\cdots+m_t=\deg B=r-1,
\]
which is \eqref{eq:extremizer-partition}.

It remains to prove the exact fiber counts.  In Case 1 of the zero-bound
proof, a fiber with \(B(\alpha)\ne0\) contributes at most one exact
zero, while a root \(\alpha_i\) of multiplicity \(m_i\) contributes at
most \(2m_i+1\) zeros of \(F\).  Summing these individual upper bounds now
gives exactly
\[
 p-t+\sum_{i=1}^t(2m_i+1)
 =p+2\sum_{i=1}^t m_i
 =p+2r-2
 =N.
\]
Since the sum of the fiberwise upper bounds equals the actual total,
every fiberwise upper bound must itself be attained.  Hence each
exceptional fiber \(a=\alpha_i\) contains exactly \(2m_i+1\) zeros and
every nonexceptional finite fiber contains exactly one.  This proves
(3).

Finally, the computational basis was chosen to have minimum support
\(k\).  Therefore every finite basis has support at least \(k\), or,
equivalently, has at most \(p-k\) zero coefficients.  Applying this to the corresponding exceptional basis
\(\mathcal B_{A_i}\), where \(A_i=-\alpha_i\), gives
\[
 2m_i+1\le p-k,
\]
which is \eqref{eq:extremizer-minimum-support-constraint}.  If
\(k\ge4\), then \(r-1\ge1\), so at least one exceptional root exists
and its multiplicity is at least one.  Hence \(p-k\ge3\), i.e.
\(k\le p-3\).  This proves (4) and completes the proof.
\end{proof}

\section{Classification and enumeration of extremizers}
\label{sec:complete-extremizers}

We now turn the equality constraints into a complete parametrization of the
extremizing rays in every prime dimension \(p\ge5\).

Put
\[
 \mathcal P:=\F_p^\times/\{\pm1\},
 \qquad
 \kappa:=|\mathcal P|=\frac{p-1}{2}.
\]
We write \([t]\) for the class of \(t\in\F_p^\times\).

\begin{lemma}[Universal constraint matrix]
\label{lem:complete-constraint-matrix}
Let
\[
 2\le r\le\frac{p-1}{2}.
\]
Choose
\[
 \mathcal R=\{[\rho_1],\ldots,[\rho_r]\}
 \in\binom{\mathcal P}{r}
\]
and
\[
 \mathcal C\in
 \binom{\F_p\times\mathcal P}{r-1}.
\]
Choose representatives \(\rho_j\in\F_p^\times\), and for each
\((\alpha,[e])\in\mathcal C\) choose \(e\in\F_p^\times\).  Define
\begin{equation}
 M_{\mathcal C,\mathcal R}
 =
 \left(
 \omega^{\alpha\rho_j^2}
 \bigl(\omega^{e\rho_j}-\omega^{-e\rho_j}\bigr)
 \right)_{
 (\alpha,[e])\in\mathcal C,\ 1\le j\le r}.
 \label{eq:complete-constraint-matrix}
\end{equation}
Then:

\begin{enumerate}
\item
\(\operatorname{rank}M_{\mathcal C,\mathcal R}=r-1\);
\item every maximal minor is nonzero;
\item \(\ker M_{\mathcal C,\mathcal R}\) is generated by a vector
\[
 q=(q_1,\ldots,q_r)
\]
with \(q_j\ne0\) for every \(j\).
\end{enumerate}

Changing a representative \(e\) to \(-e\) multiplies the corresponding
row by \(-1\) and leaves the kernel unchanged.  Changing one radius
representative \(\rho_j\) to \(-\rho_j\) multiplies the corresponding column
by \(-1\).  More generally, if \(S\) is the resulting diagonal sign matrix,
then
\[
 M'_{\mathcal C,\mathcal R}=M_{\mathcal C,\mathcal R}S,
 \qquad
 \ker M'_{\mathcal C,\mathcal R}=S^{-1}\ker M_{\mathcal C,\mathcal R}.
\]
Thus the kernel ray is covariant, rather than literally invariant, under
changes of the radius representatives.  In the state construction of
Theorem~\ref{thm:complete-extremizer-classification}, this column sign is
compensated by exchanging the two support labels \(c\pm\rho_j\); hence the
resulting projective state is independent of all representative choices.
\end{lemma}

\begin{proof}
Suppose a maximal minor vanishes.  Then there is a nonzero vector
\[
 u=(u_1,\ldots,u_r)
\]
supported on at most \(r-1\) columns such that
\[
 M_{\mathcal C,\mathcal R}u=0.
\]
Let \(s\le r-1\) be the number of nonzero coordinates of \(u\), and
form the auxiliary parabolic sum
\[
 G(a,b)
 =
 \sum_{j:u_j\ne0}
 u_j\omega^{a\rho_j^2}
 \bigl(\omega^{b\rho_j}-\omega^{-b\rho_j}\bigr).
\]
Its actual frequency support on the parabola has size \(2s\).  Since
\[
 s\le r-1\le\frac{p-3}{2},
\]
we have \(2s\le p-3\), so Theorem~\ref{thm:parabolic-zero-rigorous} applies
to this auxiliary sum.

For every \(a\in\F_p\),
\[
 G(a,0)=0.
\]
Thus \(G\) has \(p\) central zeros.  For every
\((\alpha,[e])\in\mathcal C\), the row relation gives
\[
 G(\alpha,e)=0,
\]
and oddness in \(b\) also gives
\[
 G(\alpha,-e)=0.
\]
Because the \(r-1\) elements of \(\mathcal C\) are distinct as pairs
\((\alpha,[e])\), these are \(2(r-1)\) distinct noncentral zeros.
Hence
\[
 \#Z(G)\ge p+2r-2.
\]

Theorem~\ref{thm:parabolic-zero-rigorous}, applied to the actual support
size \(2s\), gives
\[
 \#Z(G)\le p+2s-2\le p+2r-4,
\]
a contradiction.  Hence every maximal minor is nonzero.  The remaining
assertions follow from the cofactor formula for the one-dimensional
kernel.
\end{proof}

\begin{theorem}[Classification of all extremizers]
\label{thm:complete-extremizer-classification}
Let \(p\ge5\) be prime and let
\[
 \mathcal S_p(\psi)=p^2-p+2.
\]
Choose one standard MUB in which \(\psi\) has minimum support, fix a
standard-MUB-permuting unitary that sends it to the computational basis,
and work with the resulting normalized projective state.

Then the minimum support is
\begin{equation}
 k=2r
 \qquad\text{for some}\qquad
 1\le r\le\frac{p-3}{2}.
 \label{eq:allowed-minimum-supports}
\end{equation}
There are unique data
\begin{equation}
 c,\nu\in\F_p,
 \qquad
 \mathcal R\in\binom{\mathcal P}{r},
 \qquad
 \mathcal C\in
 \binom{\F_p\times\mathcal P}{r-1},
 \label{eq:complete-extremizer-data}
\end{equation}
such that, with
\[
 m_\alpha
 :=
 \#\{[e]\in\mathcal P:(\alpha,[e])\in\mathcal C\},
\]
one has the capacity inequalities
\begin{equation}
 m_\alpha
 \le
 M_{p,r}:=
 \frac{p-2r-1}{2}
 \qquad(\alpha\in\F_p).
 \label{eq:complete-capacity}
\end{equation}

For \(r=1\), the set \(\mathcal C\) is empty and we use the convention
\(q_1=1\).  For \(r\ge2\), write
\[
 \mathcal R=\{[\rho_1],\ldots,[\rho_r]\}
\]
and choose representatives \(\rho_j\) and representatives \(e\) for the
classes occurring in \(\mathcal C\), as in
Lemma~\ref{lem:complete-constraint-matrix}.  Let
\[
 q=(q_1,\ldots,q_r)
\]
generate the corresponding kernel of
\(M_{\mathcal C,\mathcal R}\).  After multiplication of
\(\psi\) by one nonzero scalar,
\begin{equation}
 \begin{aligned}
 \psi_{c+\rho_j}
 &=
 q_j\,\omega^{-\nu(c+\rho_j)},\\
 \psi_{c-\rho_j}
 &=
 -q_j\,\omega^{-\nu(c-\rho_j)}
 \qquad(1\le j\le r),
 \end{aligned}
 \label{eq:complete-state-formula}
\end{equation}
and all remaining computational coordinates vanish.  This projective state
is independent of the representative choices: changing an \(e\)-representative
only changes a row sign, while replacing \(\rho_j\) by \(-\rho_j\) changes
\(q_j\) to \(-q_j\) and simultaneously exchanges the two support labels
\(c+\rho_j\) and \(c-\rho_j\), leaving the coordinates in
\eqref{eq:complete-state-formula} unchanged after that identification.

For the parabolic sum
\[
 F(a,b)=\sum_x\psi_x\omega^{a x^2+b x},
\]
the complete zero set in \(F\)-coordinates is
\begin{equation}
 \begin{aligned}
 b&=\nu-2ac
 &&(a\in\F_p),\\
 b&=\nu-2\alpha c\pm e
 &&((\alpha,[e])\in\mathcal C).
 \end{aligned}
 \label{eq:complete-zero-set}
\end{equation}
Equivalently, in the corresponding MUB labels \((A,\beta)\), the zero
coefficients occur exactly at
\begin{equation}
 \begin{aligned}
 \beta&=-\nu-2Ac
 &&(A\in\F_p),\\
 A&=-\alpha,\qquad
 \beta=-\nu+2\alpha c\pm e
 &&((\alpha,[e])\in\mathcal C).
 \end{aligned}
 \label{eq:complete-physical-zero-set}
\end{equation}
Thus the \(F\)-fiber \(a=\alpha\) has exactly
\[
 1+2m_\alpha
\]
zeros; equivalently, the MUB \(\mathcal B_A\) contains exactly
\(1+2m_{-A}\) zero coefficients.

Conversely, every choice of data
\eqref{eq:complete-extremizer-data} satisfying
\eqref{eq:complete-capacity} gives, through
\eqref{eq:complete-state-formula}, an extremizing state whose minimum support
is \(2r\).  Once the MUB-permuting normalization above has been fixed,
the data are unique for the resulting normalized projective state.
\end{theorem}

\begin{proof}
We first treat minimum support two.  Suppose
\[
 R=\{u,v\},
 \qquad u\ne v.
\]
The computational support has size two.  By the pairwise support
inequality, every one of the other \(p\) MUBs has support at least \(p-1\).
Since the total support equals
\[
 p^2-p+2=2+p(p-1),
\]
the sum of those \(p\) remaining support sizes is exactly \(p(p-1)\).
Consequently every one of them has support exactly \(p-1\).  In particular
the Fourier basis has a zero.  Therefore, for some \(b_0\in\F_p\),
\[
 \psi_u\omega^{b_0u}
 +
 \psi_v\omega^{b_0v}
 =0.
\]
Thus
\[
 \frac{\psi_v}{\psi_u}
 =
 -\omega^{b_0(u-v)}.
\]
Putting
\[
 c=\frac{u+v}{2},
 \qquad
 \rho=\frac{u-v}{2},
\]
and absorbing one nonzero scalar, this is exactly
\eqref{eq:complete-state-formula} with \(r=1\), empty
\(\mathcal C\), and a suitable \(\nu\).  Direct substitution gives the
zero line
\[
 b=\nu-2ac.
\]
This is the asserted classification for \(r=1\).

Assume from now on that the minimum support is at least four.  By the
parity statement in Theorem~\ref{thm:extremizer-structure}, it equals
\(2r\) for some \(r\ge2\).  The same theorem gives
\begin{equation}
 Q_0(a,b)=B(a)\bigl(b+L(a)\bigr),
 \qquad
 \deg B=r-1,
 \qquad
 \deg L\le1,
 \label{eq:complete-leading-pencil}
\end{equation}
where \(B\) splits completely over \(\F_p\).  If its distinct roots are
\[
 \alpha_1,\ldots,\alpha_t
\]
with multiplicities
\[
 m_1,\ldots,m_t,
\]
then
\begin{equation}
 m_1+\cdots+m_t=r-1,
 \qquad
 t\le r-1,
 \label{eq:complete-partition}
\end{equation}
the fiber \(a=\alpha_i\) contains exactly \(2m_i+1\) zeros of \(F\), and
every nonexceptional finite fiber contains exactly one zero of \(F\).

Write
\[
 -L(a)=\mu a+\nu.
\]
Every nonexceptional fiber therefore has its unique zero at
\[
 b=\mu a+\nu.
\]
Put
\[
 d_x:=\psi_x\omega^{\nu x},
 \qquad
 h(x):=x^2+\mu x,
\]
and
\[
 g(t):=
 \sum_{\substack{x\in R\\h(x)=t}}d_x.
\]
We use the Fourier-transform convention
\[
 \widehat g(a):=\sum_{t\in\F_p}g(t)\omega^{at}.
\]
For every nonexceptional \(a\),
\[
 \widehat g(a)=0.
\]
Hence
\begin{equation}
 |\supp\widehat g|
 \le t
 \le r-1.
 \label{eq:complete-ghat-support}
\end{equation}

Let
\[
 c=-\frac{\mu}{2}.
\]
Then
\[
 h(x)=(x-c)^2-c^2.
\]
The set of squares in \(\F_p\), including \(0\), has
\((p+1)/2\) elements.  Consequently
\begin{equation}
 |\supp g|
 \le|\operatorname{im}h|
 =\frac{p+1}{2}.
 \label{eq:complete-g-support}
\end{equation}
If \(g\ne0\), Tao's prime Fourier uncertainty principle gives
\[
 p+1
 \le
 |\supp g|+|\supp\widehat g|.
\]
But the minimum support is at most \(p-1\), so
\[
 r\le\frac{p-1}{2},
\]
and therefore
\[
 |\supp g|+|\supp\widehat g|
 \le
 \frac{p+1}{2}+(r-1)
 \le
 \frac{p+1}{2}+\frac{p-3}{2}
 =p-1,
\]
a contradiction.  Hence
\begin{equation}
 g\equiv0.
 \label{eq:complete-g-zero}
\end{equation}

The fibers of \(h\) are the orbits of the involution
\[
 x\longmapsto2c-x.
\]
The fixed fiber is \(\{c\}\).  If \(c\in R\), then
\(g(h(c))=d_c\ne0\), contradicting \eqref{eq:complete-g-zero}; hence
\(c\notin R\).  Every other fiber has the form
\(\{c-\rho,c+\rho\}\).  If such a fiber met \(R\) in exactly one point,
then its contribution to \(g\) would be a single nonzero coefficient,
again contradicting \(g\equiv0\).  Thus every occupied nonfixed fiber is
occupied at both points, and on that fiber the equation \(g=0\) reads
\[
 d_{c-\rho}+d_{c+\rho}=0.
\]
Since \(|R|=2r\), there are exactly \(r\) occupied two-point fibers.
Therefore
\[
 R=
 \{c\pm\rho_1,\ldots,c\pm\rho_r\}
\]
for distinct radius classes
\[
 \mathcal R=
 \{[\rho_1],\ldots,[\rho_r]\}\subset\mathcal P,
\]
and
\begin{equation}
 d_{c-\rho_j}=-d_{c+\rho_j}
 \qquad(1\le j\le r).
 \label{eq:complete-antisymmetry}
\end{equation}
Put
\[
 q_j:=d_{c+\rho_j}.
\]
This gives \eqref{eq:complete-state-formula}.

The antisymmetry produces one central zero in every \(F\)-fiber:
\[
 F(a,\nu-2ac)=0
 \qquad(a\in\F_p).
\]
Equivalently, every finite MUB \(\mathcal B_A\) has the
central zero at \(\beta=-\nu-2Ac\).
For
\[
 H_a(t):=F(a,\nu-2ac+t),
\]
the contribution of the pair \(c\pm\rho_j\), using
\eqref{eq:complete-state-formula}, is
\[
 q_j\omega^{-ac^2+a\rho_j^2+ct}
 \bigl(\omega^{t\rho_j}-\omega^{-t\rho_j}\bigr).
\]
Replacing \(t\) by \(-t\) changes the parenthesis by a factor \(-1\)
and changes \(\omega^{ct}\) to \(\omega^{-ct}\).  Summing over \(j\)
gives
\begin{equation}
 H_a(-t)
 =
 -\omega^{-2ct}H_a(t).
 \label{eq:complete-fiber-symmetry}
\end{equation}
Thus every fiber zero set is symmetric about \(t=0\).  The
exceptional fiber \(a=\alpha_i\), which has \(2m_i+1\) zeros, therefore
has the form
\[
 t=0,\quad
 t=\pm e_{i,1},\ldots,\pm e_{i,m_i},
\]
where the classes \([e_{i,\ell}]\) are distinct.  Define
\[
 \mathcal C
 :=
 \{(\alpha_i,[e_{i,\ell}]):
 1\le i\le t,\ 1\le\ell\le m_i\}.
\]
Then
\[
 |\mathcal C|
 =
 \sum_i m_i
 =
 r-1.
\]
The noncentral zero equations are exactly
\[
 M_{\mathcal C,\mathcal R}q=0.
\]
Lemma~\ref{lem:complete-constraint-matrix} shows that this kernel ray is
unique and that all \(q_j\ne0\).

Since the computational basis was chosen to have minimum support
\(2r\), the MUB \(\mathcal B_A\), with \(A=-\alpha\), corresponding
to every exceptional \(F\)-fiber \(a=\alpha\) has support at least \(2r\):
\[
 p-(2m_\alpha+1)\ge2r.
\]
This is exactly \eqref{eq:complete-capacity}.  Since
\(|\mathcal C|=r-1>0\), some \(m_\alpha\ge1\).  Hence
\[
 M_{p,r}\ge1,
\]
which implies
\[
 2r\le p-3.
\]
Together with the already treated case \(r=1\), this proves
\eqref{eq:allowed-minimum-supports}.

The data are uniquely recoverable from the resulting normalized
projective state.  Indeed,
\[
 c=\frac{1}{2r}\sum_{x\in R}x
\]
because \(2r<p\); the support determines the radius classes
\(\mathcal R\).  Each \(F\)-fiber zero set is a nonempty proper subset of \(\F_p\)
invariant under reflection about its displayed \(b\)-center, and that center
is unique: invariance under reflections about two distinct centers would
imply invariance under a nonzero translation, hence under all translations
of the prime cyclic group, forcing the zero set to be all of \(\F_p\), which
is impossible for a nonzero state.  Thus the centers of the symmetric
\(F\)-fiber zero sets determine the affine line \(b=\nu-2ac\), hence
\(\nu\); and the noncentral symmetric zero pairs determine \(\mathcal C\).
By \eqref{eq:F-to-physical-labels}, the corresponding MUB center in
\(\mathcal B_A\) is \(\beta=-\nu-2Ac\).

Conversely, choose data satisfying
\eqref{eq:complete-extremizer-data} and
\eqref{eq:complete-capacity}.  For \(r=1\), the formula gives a
two-point state and direct substitution shows one zero in every finite
basis, so the state is sharp.

Let \(r\ge2\).  By
Lemma~\ref{lem:complete-constraint-matrix}, the constraint matrix has a
unique kernel ray generated by \(q\) with no zero coordinate.  Define
\(\psi\) by \eqref{eq:complete-state-formula}.  Its computational
support has size \(2r\).

Centered antisymmetry gives the \(p\) central zeros
\[
 b=\nu-2ac.
\]
Each constraint \((\alpha,[e])\in\mathcal C\) gives the two additional
zeros
\[
 b=\nu-2\alpha c\pm e.
\]
Thus
\[
 \#Z(F)\ge
 p+2(r-1)
 =
 p+2r-2.
\]
The parabolic zero theorem gives the reverse inequality for support
\(2r\):
\[
 \#Z(F)\le p+2r-2.
\]
Hence the displayed points are exactly all zeros of \(F\), and under
\eqref{eq:F-to-physical-labels} they correspond bijectively to all zero
coefficients in the finite MUBs.  Therefore
\[
 \mathcal S_p(\psi)
 =
 2r+p^2-(p+2r-2)
 =
 p^2-p+2.
\]
The \(F\)-fiber \(a=\alpha\) has exactly \(1+2m_\alpha\) zeros, so the
corresponding MUB \(\mathcal B_A\), where \(A=-\alpha\), has support
\[
 p-(1+2m_\alpha)\ge2r
\]
by \eqref{eq:complete-capacity}.  Every finite MUB corresponding
to a nonexceptional \(F\)-fiber has support \(p-1>2r\).  Therefore the chosen computational basis is indeed
a minimum-support basis.  This proves the converse.
\end{proof}

\begin{corollary}[Displaced parity and a discrete-Wigner minimum]
\label{cor:wigner-minimum}
Let \(\psi\) be a norm-one extremizer, normalized as in
Theorem~\ref{thm:complete-extremizer-classification}, with parameters
\(c,\nu\).  Define the displaced-parity involution
\begin{equation}
 (\mathsf A_{c,\eta}f)(x)
 =
 \omega^{2\eta(x-c)}f(2c-x).
 \label{eq:displaced-parity}
\end{equation}
Then
\begin{equation}
 \mathsf A_{c,-\nu}\psi=-\psi.
 \label{eq:negative-parity-extremizer}
\end{equation}
Consequently, in the odd-prime discrete-Wigner convention used here, in which
\(\mathsf A_{c,\eta}\) is the phase-point operator \cite{WoottersWigner,ZhuWigner},
\begin{equation}
 W_\psi(c,-\nu)
 =
 \frac1p\langle\psi,\mathsf A_{c,-\nu}\psi\rangle
 =-\frac1p.
 \label{eq:wigner-minimum}
\end{equation}
This is the smallest possible value at a single phase-space point in this
normalization.
\end{corollary}

\begin{proof}
For a support point \(x=c+\rho_j\), formula
\eqref{eq:complete-state-formula} gives
\[
 \begin{aligned}
 (\mathsf A_{c,-\nu}\psi)(c+\rho_j)
 &=
 \omega^{-2\nu\rho_j}\psi(c-\rho_j)\\
 &=
 -q_j\omega^{-2\nu\rho_j}\omega^{-\nu(c-\rho_j)}\\
 &=
 -q_j\omega^{-\nu(c+\rho_j)}
 =-\psi(c+\rho_j).
 \end{aligned}
\]
The calculation at \(c-\rho_j\) is identical, and both sides vanish off
the support.  Hence \eqref{eq:negative-parity-extremizer} holds.  Moreover,
\(\mathsf A_{c,\eta}^2=I\), and \(\mathsf A_{c,\eta}\) is unitary, so its
eigenvalues are \(\pm1\).  For a norm-one state, every expectation
\(\langle\psi,\mathsf A_{c,\eta}\psi\rangle\) therefore lies in
\([-1,1]\), which proves the final assertion.
\end{proof}

The implication in Corollary~\ref{cor:wigner-minimum} is only a necessary
property of total-support extremizers: belonging to a negative-parity
eigenspace does not by itself imply total-support extremality.

\begin{example}[A four-point extremizer in dimension seven]
\label{ex:p7-four-point}
Let \(p=7\), take \(c=\nu=0\), and choose
\[
 \mathcal R=\{[1],[2]\},
 \qquad
 \mathcal C=\{(0,[1])\}.
\]
Put
\[
 s_j:=\omega^j-\omega^{-j}.
\]
With representatives \(\rho_1=1\), \(\rho_2=2\), and \(e=1\), the
constraint matrix is the single row \((s_1,s_2)\), so its kernel is
generated by
\[
 q=(s_2,-s_1).
\]
The classification theorem therefore gives the explicit state
\begin{equation}
 \psi
 =
 s_2(e_1-e_{-1})-s_1(e_2-e_{-2}).
 \label{eq:p7-four-point-state}
\end{equation}
Its computational support is four.  Formula
\eqref{eq:complete-zero-set} gives the central zero \(b=0\) in every
finite fiber and the two additional zeros \(b=\pm1\) in the fiber
\(a=0\).  Hence the support sizes across the eight standard MUBs are
\[
 (4,4,6,6,6,6,6,6),
\]
up to permutation of the basis labels, and their sum is \(44\).  Thus
\eqref{eq:p7-four-point-state} is a fully analytic four-point extremizer of
the type that was previously obtained numerically in dimension seven.
\end{example}

\begin{corollary}[Possible minimum supports]
\label{cor:possible-minimum-supports}
For \(p\ge5\), the minimum support of an extremizing state is exactly one of
\[
 2,4,6,\ldots,p-3.
\]
Every value in this list occurs.
\end{corollary}

\begin{proof}
Necessity follows from
Theorem~\ref{thm:complete-extremizer-classification}.  Conversely, let
\(2r\le p-3\).  Then
\[
 M_{p,r}=\frac{p-2r-1}{2}\ge1.
\]
Choose \(r\) radius classes and choose the \(r-1\) constraints in
distinct \(a\)-fibers, so every occupancy \(m_\alpha\) is at most one.
The capacity condition holds, and
Theorem~\ref{thm:complete-extremizer-classification} produces a sharp
state of minimum support \(2r\).
\end{proof}

\begin{corollary}[Sharpness at every even support]
\label{cor:even-parabolic-sharpness}
For every prime $p\ge5$ and every integer \(r\) with
\[
 1\le r\le\frac{p-1}{2},
\]
there exist $R\subset\mathbb F_p$ with $|R|=2r$ and nonzero complex
coefficients $(c_x)_{x\in R}$ such that
\[
 \#Z\!\left(\sum_{x\in R}c_x\omega^{ax^2+bx}\right)=p+2r-2.
\]
Hence the even-support bound in Theorem~\ref{thm:intro-parabolic} is sharp
throughout its full range.
\end{corollary}

\begin{proof}
First suppose \(2r\le p-3\).  By
Corollary~\ref{cor:possible-minimum-supports}, choose a sharp MUB state whose
minimum support is \(2r\), and normalize one minimum-support basis to
\(\mathcal B_\infty\).  If \(N\) is the number of zero coefficients among
the \(p\) finite quadratic bases, then
\[
 p^2-p+2=2r+p^2-N,
\]
so \(N=p+2r-2\).  By \eqref{eq:mub-parabola-label-map} and the bijection
\((A,\beta)\mapsto(-A,-\beta)\), these are exactly the zeros of the
associated parabolic sum.

It remains to treat the endpoint \(2r=p-1\).  Take
\[
 R=\F_p^\times,
 \qquad
 c_x=1-\omega^{-x}\quad(x\in R).
\]
All these coefficients are nonzero.  Since the coefficient at \(x=0\)
would be zero, we may extend the sum to all \(x\in\F_p\) and write
\[
 F(a,b)
 =
 \sum_{x\in\F_p}\omega^{ax^2+bx}
 -
 \sum_{x\in\F_p}\omega^{ax^2+(b-1)x}.
\]
For \(a=0\), character orthogonality gives
\[
 F(0,b)=p\,\mathbf 1_{\{b=0\}}-p\,\mathbf 1_{\{b=1\}},
\]
so the fiber \(a=0\) contains exactly \(p-2\) zeros.  For \(a\ne0\),
completing the square gives
\[
 \sum_{x\in\F_p}\omega^{ax^2+bx}
 =
 G(a)\omega^{-b^2/(4a)},
 \qquad
 G(a):=\sum_{x\in\F_p}\omega^{ax^2}\ne0.
\]
Hence \(F(a,b)=0\) if and only if
\[
 b^2=(b-1)^2,
\]
that is, \(b=\tfrac12\).  Thus every one of the \(p-1\) nonzero
\(a\)-fibers contributes exactly one zero, and
\[
 |Z(F)|=(p-2)+(p-1)=2p-3=p+(p-1)-2.
\]
This is equality in the even-support bound for \(2r=p-1\).
\end{proof}

\begin{corollary}[Enumeration of extremizers]
\label{cor:complete-extremizer-count}
For
\[
 1\le r\le\frac{p-3}{2},
\]
put
\[
 M=M_{p,r}=\frac{p-2r-1}{2},
 \qquad
 \kappa=\frac{p-1}{2},
\]
and define
\begin{equation}
 P_{p,r}(z,u)
 :=
 \sum_{j=0}^{M-1}\binom{\kappa}{j} z^j
 +
 u\binom{\kappa}{M}z^M.
 \label{eq:counting-polynomial}
\end{equation}
Then the number \(N_{2r}(p)\) of sharp projective rays whose minimum
support is \(2r\) is
\begin{equation}
 N_{2r}(p)
 =
 (p+1)p^2\binom{\kappa}{r}
 \,[z^{r-1}]
 \int_0^1
 P_{p,r}(z,u)^p\,du.
 \label{eq:complete-count-formula}
\end{equation}
The integral is coefficientwise in the polynomial variable \(u\).
Consequently the total number of extremizing rays is
\begin{equation}
 N_{\mathrm{sharp}}(p)
 =
 \sum_{r=1}^{(p-3)/2}N_{2r}(p).
 \label{eq:total-sharp-count}
\end{equation}
\end{corollary}

\begin{proof}
Fix one standard MUB and designate it as the chosen minimum-support
basis.  By the classification theorem, a normalized extremizer of
minimum support \(2r\) is determined uniquely by:

\begin{enumerate}
\item a center \(c\in\F_p\);
\item a parameter \(\nu\in\F_p\);
\item \(r\) radius classes chosen from the \(\kappa=(p-1)/2\) classes in
\(\mathcal P\);
\item a constraint set
\[
 \mathcal C\subset\F_p\times\mathcal P,
 \qquad
 |\mathcal C|=r-1,
\]
whose occupancy in each \(a\)-fiber is at most \(M\).
\end{enumerate}

For a fixed \(a\), choosing exactly \(j\) constraints in that fiber can
be done in
\[
 \binom{\kappa}{j}
\]
ways.  The polynomial
\[
 P_{p,r}(z,u)
\]
records this choice, with \(z\) marking the number of constraints and
\(u\) marking a saturated fiber, that is, a fiber with occupancy
exactly \(M\).  Hence
\[
 [z^{r-1}u^s]P_{p,r}(z,u)^p
\]
counts the constraint sets with exactly \(s\) saturated finite fibers.

For a finite fiber with occupancy \(m_\alpha\), the classification
theorem gives support
\[
 p-(1+2m_\alpha).
\]
By definition
\[
 M=\frac{p-2r-1}{2},
\]
so this support equals \(2r\) if and only if \(m_\alpha=M\); if
\(m_\alpha<M\), it is strictly larger than \(2r\).  Thus the saturated \(F\)-fibers are exactly those whose corresponding
MUBs \(\mathcal B_A\), with \(A=-\alpha\), also have minimum support.  A
constraint set with exactly \(s\) saturated fibers therefore produces exactly
\[
 1+s
\]
minimum-support bases: the chosen computational basis and those \(s\) finite
bases.  Consequently a count in which one designates a minimum-support basis
counts the same projective ray exactly \(1+s\) times.

Fix once and for all, for each standard MUB, one MUB-permuting unitary
that sends it to the computational basis.  This gives a bijection between
rays with that designated minimum-support basis and the normalized rays
counted above.  Hence there are \(p+1\) choices for the designated basis,
and the remaining choices contribute
\[
 p^2\binom{\kappa}{r}.
\]
Therefore
\[
 N_{2r}(p)
 =
 (p+1)p^2\binom{\kappa}{r}
 \sum_{s\ge0}
 \frac{1}{s+1}
 [z^{r-1}u^s]P_{p,r}(z,u)^p.
\]
Since
\[
 \frac1{s+1}=\int_0^1u^s\,du,
\]
this is exactly \eqref{eq:complete-count-formula}.
\end{proof}

\begin{remark}[Low-dimensional counts]
\label{rem:low-dimensional-counts}
Evaluating the closed formula \eqref{eq:complete-count-formula} gives
\begin{center}
\begin{tabular}{ccc}
\toprule
prime dimension \(p\) & minimum support & number of extremizing rays \\
\midrule
5 & 2 & 300 \\
7 & 2 & 1176 \\
7 & 4 & 12348 \\
\bottomrule
\end{tabular}
\end{center}
For the four-point case \(p=7\), \(r=2\), one has \(\kappa=3\) and
\(M=1\).  A fixed designated minimum-support basis contributes
\(7^2\binom{3}{2}(7\cdot3)\) parameter choices before correcting for the
second minimum-support basis, and every such ray is counted twice.  Hence
\[
 N_4(7)
 =
 \frac{8\cdot7^2\binom{3}{2}\cdot21}{2}
 =12348.
\]
Thus dimension seven has \(1176+12348=13524\) extremizing projective
rays in total.  These numbers count rays for the fixed standard complete set
of MUBs; no quotient by the Clifford group or any other symmetry group is
taken.  The entries in the table are evaluations of the proved enumeration
formula, not separate computational evidence for the classification.
\end{remark}

\begin{remark}[Geometry of equality]
\label{rem:complete-equality-geometry}
The classification shows that every extremizing state is governed by two
reflection structures.

First, its minimum-support coordinates occur in pairs
\[
 c\pm\rho_j
\]
and the dephased coefficients are odd under reflection about \(c\).
Second, in \(F\)-coordinates the fiberwise zero sets are symmetric about
the affine line
\[
 b=\nu-2ac.
\]
Equivalently, in the finite-basis labels the centers lie on
\[
 \beta=-\nu-2Ac.
\]
The noncentral zero pairs form an arbitrary admissible set
\[
 \mathcal C\subset\F_p\times\mathcal P
\]
subject only to the capacity condition
\eqref{eq:complete-capacity}; the universal nonvanishing of the
constraint-matrix maximal minors then produces the unique extremizing
ray.
\end{remark}

\section{Conclusion}
\label{sec:conclusion}

The complete-MUB total-support problem is controlled by a zero-count question for sparse Fourier data on the finite parabola. The cyclotomic reduction and truncated-jet argument prove the required zero theorem for arbitrary complex coefficients, and the equality case propagates enough structure to classify every extremizing state. In particular, the same mechanism gives the sharp value of the total support, the complete list of possible minimum-support sizes, and a closed enumeration formula for the extremizing projective rays.

The main point is that the complete set of quadratic-phase bases must be treated simultaneously. Pairwise Fourier uncertainty supplies useful background bounds, but the sharp result comes from the common parabolic parametrization of all finite bases. This is also why the equality theory is unusually rigid: the zero pattern of one sparse parabolic Fourier sum records the vanishing coordinates across the entire family of bases. The classification also has a direct phase-space consequence: every normalized extremizer is a negative eigenstate of an appropriate displaced-parity operator and therefore attains the single-point discrete-Wigner minimum \(-1/p\) at one phase-space point.

\section*{Data availability}
No new data were created or analysed in this study.

\medskip
\noindent\textbf{Declaration of AI-assisted tools.}
OpenAI's ChatGPT was used for literature retrieval, proof verification, and language editing.

\end{document}